\documentclass{amsart}
\usepackage{graphicx}
\usepackage{stmaryrd}
\usepackage{enumerate}
\usepackage{txfonts}
\usepackage{color}
\usepackage[dvipsnames]{xcolor}

\usepackage[T1]{fontenc}
\usepackage{lmodern}

\usepackage[colorlinks=true,
            linkcolor=violet,
            urlcolor=Aquamarine,
            citecolor=YellowOrange]{hyperref}

\newtheorem{thm}{Theorem}[section]
\newtheorem{cor}[thm]{Corollary}
\newtheorem{lem}[thm]{Lemma}
\newtheorem{prop}[thm]{Proposition}
\newtheorem{defn}[thm]{Definition}
\newtheorem{rem}[thm]{Remark}

\newcommand{\abs}[1]{\left\vert#1\right\vert}
\newcommand{\set}[1]{\left\{#1\right\}}
\newcommand{\Real}{\mathbb R}

\newcommand{\pfrac}[2]{\frac{\partial #1}{\partial #2}}

\definecolor{darkgreen}{rgb}{0,0.5,0}
\definecolor{darkred}{rgb}{0.7,0,0}

\title[expansions of ricci flat ALE metrics]{
\bf
on expansions of ricci flat ALE metrics \\
 in harmonic coordinates about the infinity
}
\author{Youmin Chen}
\address{School of Mathematical Sciences, University of Science and Technology of China}%
\email{chym189@mail.ustc.edu.cn}%
\date{\today~(@\the\time mpm)}

\begin{document}
\maketitle

\begin{abstract}
In this paper, we study the expansions of Ricci flat metrics in harmonic coordinates about the infinity of ALE (asymptotically local Euclidean) manifolds.
\end{abstract}

\section{Introduction}
A smooth complete orientable Riemannian manifold $(M^{n},g)$ $(n \geq 3)$ is called an \textit{Asymptotically Locally Euclidean} (ALE) manifold  of order $ \tau>0$ if for some compact
$K \subset M^{n} $, $M^{n} \setminus K $ consists of
a finite number of components $ E_{1}$,...,$E_{k}$ such that for each $E_{i}$ there exists a finite subgroup $ \Gamma_{i}\subset O(n)$ and a $\mathcal{C}^{\infty} $  diffeomorphism
$\Phi :E_{i} \rightarrow (\mathbb{R}^{n}\setminus B(0;R))/\Gamma_{i}$ such that $\varphi= \Phi^{-1}\circ proj $ satisfies
(where $proj$ is the natural projection of $\mathbb{R}^{n}$ to $\mathbb{R}^{n} /\Gamma_{i}$ )
$$ (\varphi^{\ast}g)_{ij}(z)=\delta_{ij}+O(|z|^{-\tau}),
\ \ \ \ \ \partial_{k}(\varphi^{\ast}g)_{ij}(z)=O(|z|^{-\tau-1}),$$
$$ \frac{|\partial_{k}(\varphi^{\ast}g)_{ij}(z)-\partial_{k}(\varphi^{\ast}g)_{ij}(w)|}
{|z-w|^{\alpha}}=O(min\{|z|,|w|\}^{-\tau-1-\alpha})$$
for $z,w\in \mathbb{R}^{n}\setminus B(0;R)$,
where $\partial_{k}$ denotes $\frac{\partial} {\partial{z_{k}}}$. In the above definition each $E_{i}$ is called an end of $M$, without loss of generality, all the manifolds have only one end in this paper. In the rest of the paper when we say coordinates at the infinity, we mean a diffeomorphism $\Phi$ as above. ALE manifolds play an important role in differential geometry
(e.g., as singularity models in studying convergence of Einstein manifolds with bounded Ricci curvature  \cite{anderson1989ricci,nakajima1988hausdorff,bando1990bubbling}) and in general relativity (e.g., gravitational instantons and the positive mass theorem).

In \cite{BKN} Bando, Kasue and Nakajima proved a deep result (Theorem (1.1) in \cite{BKN}, see also Theorem \ref{mainthmBKN}) which ensures the existence of harmonic coordinates $ (z_{1},\cdots,z_{n})$ at the infinity of a Riemannian manifold $(M,g)$ if it has maximal volume growth and curvature decay a bit faster than quadratic. Further assumed $(M,g)$ has $ L^{\frac{n}{2}}$ integral curvature bounds and is Ricci flat, they showed that the metric at infinity is very close to the standard Euclidean metric (Theorem (1.5) in \cite{BKN}, see also Corollary \ref{maincorBKN}), to be more precise,
\begin{equation}\label{decayorder}
	g_{ij}=\delta_{ij}+O(|z|^{1-n}),
\end{equation}
in the coordinates above,
where $|\cdot|$ is the usual vector norm on $\mathbb{R}^{n}$. (\ref{decayorder}) gives us the best decay order of such metrics, which can be seen as a first step in understanding the regularity. We want to study the full regularity in this paper. In the following we focus on Ricci flat ALE metrics. They have been studied by many authors, the readers can see \cite{ni2003ricci} and the references therein.

Similar regularity problems have been studied by many authors in various cases.
 Well known examples include the behaviour of asymptotically hyperbolic Einstein metrics \cite{graham1999volume,kichenassamy2004conjecture,chrusciel2005boundary,helliwell2008boundary,fefferman2011ambient}, the boundary regularity of minimal graphs in the hyperbolic space \cite{lin1989dirichlet,lin1989asymptotic,han2014boundary}, the boundary behavior of the complex Monge-Amp\`{e}re equation \cite{lee1982boundary}, and regularity for the singular Yamabe problem \cite{mazzeo1991regularity,korevaar1999refined}.

In all the above mentioned examples, the full regularity is obtained by proving an asymptotic expansion. This is exactly what we will do in this paper for Ricci flat ALE metrics.
 But we can not expect nice expansions just as simple as the Taylor expansion since usually there are obstacles to the higher regularity of the solutions.
 For example, in \cite{jian2014optimal,jian2013bernstein} Jian and Wang showed clearly how the singularity of the equation can effect the regularity of the solutions. And it has been known for a long time that logarithmic terms should appear in the expansions \cite{fefferman1976monge,lee1982boundary,mazzeo1991regularity}.

It is unlikely to find a unified approach to them, because every problem has its own particular form of expansion. For instance, whether log terms appear or not are different for different problems. In \cite{yin2016analysis} the author showed that the expansion for the Ricci
flow on a conical surface involves no log terms.
An interesting phenomenon was discovered in the conformally compact Einstein metric context \cite{chrusciel2005boundary} that logarithmic terms appear only when the dimension of the manifold is odd.
Similar results were discovered in \cite{yin2007boundary,han2014boundary}.
  But it is not the case for the expansion of constant mean curvature graphs in the hyperbolic space \cite{han2016boundary} where log terms appear in all dimensions except in the dimension 2. These differences are caused both by the nature of the singularity or degeneracy of the problem and by the nonlinearity of the PDE.

  In general we have no idea about the exact pattern of the expansion for a specific problem in advance, the best one can hope is that the power of logarithmic factor can be controlled in some sense. For an example in \cite{jeffres2016kahler} the authors studied Kahler-Einstein metrics with edge singularities and got an expansion for the potential function of the metric along the given divisor, in which the power of $\log \rho$ is controlled by a positive integer depending on the power of $\rho$ in every single term, where $\rho$ is the distance to the divisor. Later in \cite{yin2016expansion} this dependence was shown to be rather simple.

In order to state the main result of this paper, we introduce some necessary notations. Throughout the whole paper, $n$ denotes the dimension of the ALE manifold under consideration, which will always be greater than or equal to $4$.
The sphere coordinates is the best for our purpose. And in the rest of paper we set $(r,\omega)=(z_{1},\cdots,z_{n})=z,$ where $z$ is the harmonic coordinates above, $r=|z|$ and $\omega\in S^{n-1}$. $\mathcal{C}^{k}(\mathbb{R}^{n}\setminus B_{R})$ denotes the space of all $k$ times differentiable functions defined on $\mathbb{R}^{n}\setminus B_{R}$. For a set $A$ of functions, $\text{SPAN}(A)$ denotes the space of all finite linear combinations generated by $A$.

 For the formal expansion of Ricci flat ALE metrics, we need the following two sets of functions (for $n> 4$ and $n=4$ respectively):
\begin{eqnarray*}\label{nologterm}
	{\mathcal T} &:=&\Big\{r^\sigma G_m; \quad G_m\in \mathcal{H}_{m}(S^{n-1}), \sigma=2j -(n-2)l-k,\nonumber\\
                 & &k\geq l \geq j+1,\quad l,k,m,j,\frac{k-m}{2}\in \mathbb N\cup \set{0}\Big\},
\end{eqnarray*}
and
\begin{eqnarray*}\label{logterm1}
		\widetilde{\mathcal{T}}
                 &:=&\Big\{r^\sigma (\log r)^{i}G_m; \quad G_m\in \mathcal{H}_{m}(S^{3}), \sigma= -2l-k,\nonumber\\
                 & & \quad l\geq 1,k\geq l+i, \quad l,k,i,m ,\frac{k-m}{2} \in\mathbb N\cup \set{0}\Big\},
\end{eqnarray*}
where $\mathcal{H}_{m}(S^{n-1})$ is the linear space of all eigenfunctions of $\triangle_{S^{n-1}}$ corresponding to the eigenvalue $-m(m+n-2)$ and $\triangle_{S^{n-1}}$ is the Laplacian of the unit sphere $S^{n-1}$. We remark that the particular form of $\mathcal{T}$ and $\widetilde{\mathcal T}$ is a consequence of a careful analysis of the structure of the equation satisfied by the metric. It should be justified in Section \ref{sec:set} and Section \ref{sec:str}.

To describe the remainder of our expansion, we need an error term slightly different from the standard analysis text book. More precisely, a function $u$ defined on $\Real^n\setminus B_R$ is said to be $O(-q)$ ($q\geq0$) if and only if
\begin{equation*}
	\abs{r^j D^j u}\leq C(j)r^{-q}\qquad \text{as}\; r\rightarrow\infty,
\end{equation*}
for all $j\in \mathbb N\cup \set{0}$, where $C(j)$ is a positive number depending on $j$. In this paper $f=O(-q)$ means that $f$ is a function as above. Note that there are many functions which belong to $O(r^{-q})$ but are not $O(-q)$. For example $\frac{\sin r}{r}=O(r^{-1})$ since $\frac{\sin r}{r}\leq \frac{1}{r}$, but it is not $O(-1)$. Obviously if $f$ has the form $r^{-q}G_m$, then $f=O(-q)$. And functions of the form of $r^{-q}(\log r)^{i}G_m$ belong to $O(-q+\varepsilon)$ for any $\varepsilon>0.$
With this in mind, we define
\begin{defn} \label{defexpansionintr}
	Given a function $u\in\mathcal{C}^{\infty}(\mathbb{R}^{n}\setminus B_{R})$ and a set $A$, then $u$ is said to have an expansion with terms in $A$ up to order $-q$, $q\in\mathbb{N}$, if for some $1>\epsilon>0$ there is $u_q\in \mbox{SPAN}(A)$ such that
	\begin{equation*}
		u=u_q + O(-q-\varepsilon).
	\end{equation*}
\end{defn}

Now, we are ready to state our main result, which shows that the precise regularity depends on the dimension,
\begin{thm} \label{MAINTHM}
 Suppose that $(M,g)$ is a complete $n-$dimensional ($n\geq 4$) Riemannian manifold satisfying
 \begin{enumerate}[(a)]
	 \item Ricci flat,
	 \item maximum volume growth,
	 \item $L^{n/2}$ norm of the curvature tensor is finite,
 \end{enumerate}
  then

  (1) there are harmonic coordinates about the infinity on $(M,g)$ in which $g_{ij}-\delta_{ij}$ has an asymptotical expansion with terms in $\mathcal{T}$ ($n>4$) or $\widetilde{\mathcal{T}}$ ($n=4$) up to any order in the sense of Definition \ref{defexpansionintr},

  (2) the first $n-1$ terms ( i.e. terms of order from $r^{1-n}$ to  $r^{3-2n}$ ) in the expansion are all harmonic functions on Euclidean space $\mathbb{R}^{n}$.
\end{thm}

 From the first part it is obvious that there are no log terms in the expansion of the metrics when $n\geq 5$. But log terms may appear when $n=4$.

To get a better understanding of our expansion, we will perform the Kelvin transformation to the metric coefficients. First we recall the inversion map, defined on $\mathbb{R}^{n}\setminus \set{0}$,
\begin{equation*}
x\mapsto x^{*}=\frac{x}{|x|^{2}}.
\end{equation*}
Suppose $u$ is a function in $\mathcal{C}^{\infty}(\mathbb{R}^{n}\setminus B_{R})$, then the Kelvin transformation
(see Chapter 4 in \cite{bookharmonicfunction}) changes $u$ to be a function $K[u] $ on $B_{\frac{1}{R}}(0)\setminus \set{0} $, which is defined to be $K[u]=|x|^{2-n} u(x^{*})$, for $x\in B_{\frac{1}{R}}\setminus \set{0}$. Under this transformation the Ricci flat ALE metrics behave as follows when the dimension is even and greater than 4,
 \begin{cor}\label{coreven}
 When $n>4$ and $n$ is even,
 $K[g_{ij}-\delta_{ij}]$ can be extended to be a smooth function on $B_{\frac{1}{2R}},$ and its value at the origin $0\in B_{\frac{1}{R}}$ is zero.
\end{cor}

\vskip 0.5cm
Now let's turn to the proof of Theorem \ref{MAINTHM}. In harmonic coordinates, the Ricci flat condition gives
\begin{equation} \label{equricciflat1}
   \triangle_{g} g_{ij}= Q_{ij}(g,\partial g),
\end{equation}
where $Q_{ij}(g,\partial g)$ is a quadratic term of the first derivatives of coefficients of the metric.

  Here, we briefly outline the proof for $n>4$. It is a bootstrapping argument. First we rewrite (\ref{equricciflat1}) in the form
  \begin{equation} \label{formal}
  \triangle_{\mathbb{R}^{n}}g_{ij}=RHS_{ij},
  \end{equation}
  where $ \triangle_{\mathbb{R}^{n}}$ is the Laplace operator on the $n$
  dimensional Euclidean space.
  At the beginning we pretend that $g_{ij}$ and hence its derivatives are finite linear combinations of functions in $\mathcal{T}$, then $RHS_{ij}(g)$ is a linear combination with terms in some set $\mathcal{T}_{rhs} $, which is a subset of $\mathcal{T} $ due to some closedness properties of $\text{SPAN}(\mathcal{T}) $ (see Lemma \ref{finalset5}). The key to the whole argument is that for every term in $\mathcal{T}_{rhs} $, the corresponding Poisson equation can be solved with a solution in $\text{SPAN}(\mathcal{T}) $.  This is not trivial and depends on the nonlinear structure of the right hand side of (\ref{formal}).

  As for the error term, we shall keep track of the error term in computing $RHS_{ij}$ to find out that the order of the error increases by $n+1$. Solving the Poisson equation will reduce the order by $2$, thanks to a result of Meyers \cite{Mey} (see Lemma \ref{corpoissonestimate4}). Combining these together yields an improvement of the regularity of $g_{ij}$, so that the bootstrapping argument works.

 As a result of the above argument, when $n>4$, the structure of (\ref{formal}) enables us to avoid log terms during bootstrapping, but when $n=4$, the log terms appear from solving the Poisson equations. That's why we handle the two cases separately.

The paper is organized as follows. In Section \ref{sec:pre} we recall some basic knowledge about harmonic coordinates and the main results of \cite{BKN}. In Section \ref{sec:set} we verify $\text{SPAN}(\widetilde{\mathcal{T} })$ and $\text{SPAN}(\mathcal{T} )$ are closed under multiplication and differentiation, and in particular closed under matrix inversion in some sense.
In Section \ref{sec:sol} we solve the Poisson equations with particular terms in the right hand side, and give an error estimate
which was used to deal the remainders. In Section \ref{sec:str} we study the structure of the right hand side of (\ref{formal}) which is done to show the solvability of the Poisson equations involved in the bootstrapping argument. Finally in Section \ref{sec:pro} with all the preparations in the previous sections we give the proofs of Theorem \ref{MAINTHM} and Corollary \ref{coreven}.

\section{Preliminaries }\label{sec:pre}

\subsection{ALE manifolds} \label{ale}
In this subsection we recall the main results of \cite{BKN}.
The first is a profound theorem for manifolds with fast curvature decay and maximal volume growth.

 \begin{thm}\label{mainthmBKN} (Theorem (1.1) in \cite{BKN})
 Let $o\in M $. Suppose $(M^{n},g)$ has only one end, $ n\geq 3 $, and satisfies
\begin{equation} \label{curvadecay}
	 |Rm|\leq K\rho^{-(2+\epsilon)}
\end{equation}
for sufficiently large $\rho=d(o,*)$, ( $Rm $ denotes the curvature tensor of $g$ )

\begin{equation} \label{ricdecay}
|Ric|\leq K\rho^{-(2+\eta)} \quad \text{for sufficiently large $\rho$},
\end{equation}
and
 \begin{equation} \label{volume}
	 vol(B(o;t))\geq Vt^{n}
\end{equation}
 for all $t>0$, where $K\geq0, V>0$. Then $(M^{n},g)$ is ALE of order $\mu$ which is a real number depends on $n$, $ \eta$ ( or $ \epsilon$),
 moreover we can take $ \mu=n-1$ if $ \eta>n-1$.
 \end{thm}

Given (\ref{volume}), the condition (\ref{curvadecay}) alone ensure the existence of harmonic coordinates in which the metric is
ALE of order $\epsilon'$, for some $0<\epsilon'<\epsilon$. Then (\ref{ricdecay}) is used to improve the order in the harmonic coordinates.

The next is a crucial result we need to describe the asymptotical behaviour of Ricci flat ALE metrics.
 \begin{cor}\label{maincorBKN} (Theorem (1.5) in \cite{BKN})
 Let $(M,g)$ be a $n$ dimensional Ricci flat manifold ($n\geq4$) with
 \begin{equation*}
  vol(B(o;t))\geq Vt^{n}
 \quad \text{for some $ o\in M$,$V>0$},
\end{equation*}
and
 \begin{equation} \label{curvaintegralnorm}
	 \int_{M} |Rm|^{\frac{n}{2}}\,dV_{g}<\infty.
\end{equation}
Then $(M,g)$ is ALE of order $n-1$ in some harmonic coordinates $(\mathcal{L};\mathbb{R}^{n}\setminus B_{R})$ (large $R\in\mathbb{R}$) about the infinity.
 \end{cor}

 To get further regularity results of the metrics we need to restate Corollary \ref{maincorBKN} using the language of weighted H\"{o}lder space.

 \begin{defn}\label{defholder}
 The weighted H\"{o}lder space $ \mathcal{C}^{0,\alpha}_{\delta}$ on $ \mathbb{R}^{n}\setminus B_{R},R\in \mathbb{R}$ with weight $\delta$ is the space of functions with the following norm
 \begin{eqnarray*}
 \|u\|_{\mathcal{C}^{0,\alpha}_{\delta}}&=&\sup_{x\in\mathbb{R}^{n}\setminus B_{R}}
 \Big( r^{-\delta+\alpha}(x)\sup_{4|x-y|\leq r(x)}\frac{u(x)-u(y)}{|x-y|^{\alpha}} \Big)\\
 & &+\sup_{x\in\mathbb{R}^{n}\setminus B_{R}}\{ r^{-\delta}(x)|u(x)|\},\quad x=(r,\omega).
 \end{eqnarray*}
 \end{defn}
 The other H\"{o}lder norms like $ \|\cdot\| _{\mathcal{C}^{k,\alpha}_{\delta}} $ can be defined by
 \begin{equation*}
 \|u\|_{\mathcal{C}^{k,\alpha}_{\delta}}=\sum_{0}^{k}\|D^{j}u\|_{\mathcal{C}^{0,\alpha}_{\delta-j}}.
 \end{equation*}
 Since all functions in this paper are defined on $\mathbb{R}^{n}\setminus B_{R}$ for some large $R\in\mathbb{R}$, sometimes we just write $\mathcal{C}^{k,\alpha}_{\delta}$ for short when there is no confusion.  If $u\in\mathcal{C}^{k,\alpha}_{\delta}(\mathbb{R}^{n}\setminus B_{R})$ for all $k\in\mathbb N\cup \{0\}$, we say $u\in\mathcal{C}^{\infty}_{\delta}(\mathbb{R}^{n}\setminus B_{R}) $ for simplicity. For more about weighted H\"{o}lder space and weighted Sobolev space we refer the readers to \cite{Ba} (page 663-664).

By the remarks $(1.8)$ below Theorem $(1.5)$ in
\cite{BKN} we know that when $n \geq 4$, for $(M,g)$ which meets the requirements of Corollary \ref{maincorBKN}, it holds that
\begin{equation*}
 D^{k}(\varphi^{\ast}g)_{ij}(z)-D^{k}\delta_{ij}=O(|z|^{-(n-1)-k})
 \end{equation*}
 in the harmonic coordinates for all $k\geq 0$. Thus we conclude

 \begin{thm} \label{MAINTHM1}
 Suppose that $(M,g)$ is a complete $n-$dimensional($n\geq 4$) Riemannian manifold satisfying
 \begin{enumerate}[(a)]
	 \item Ricci flat;
	 \item maximum volume growth;
	 \item $L^{n/2}$ norm of the curvature tensor is finite,
 \end{enumerate}
  then there are harmonic coordinates about the infinity on $(M,g)$ in which $g_{ij}-\delta_{ij}$ belongs to $\mathcal{C}^{\infty}_{1-n}(\mathbb{R}^{n}\setminus B_{R})$ ($R$ is large enough).
 \end{thm}

\subsection{Harmonic coordinates}
In this subsection we state some basic facts about harmonic coordinates:

($a$) Let $g_{ij}$ be the coefficients of the metric $g$ of $M$ in terms of harmonic coordinates, then we have
  $$\sum_{i,j} g^{ij} \Gamma^{k}_{ij}=0$$
  and
 \begin{equation}
	 \triangle_{g}u=\sum_{i,j}g^{ij}\partial_{i}\partial_{j}u.
   \label{eqn:laplace oporater}
\end{equation}

($b$) Most importantly, in harmonic coordinates, the coefficients of $g$ are subject to an elliptic system
$$\triangle_{g} g_{ij}=-2 Ric_{ij} + Q_{ij}(g,\partial g),$$
  where
  \begin{equation*}
	Q_{ij}(g,\partial g)=-\sum_{k,l,p,q} g^{pq}g_{lj}\partial_{p} g_{ik}\partial_{q} g^{kl}
      -\sum_{k,l,p,q,r,s} 2g_{ik}g_{lj} g^{pq}g^{rs}\Gamma^{k}_{pr}\Gamma^{l}_{qs}.
	\label{eqn:nonliear term}
\end{equation*}
In particular if $Ric_{ij}=0$, then the equation for Ricci flat ALE  manifolds in the above harmonic coordinates $(\mathcal{L};\mathbb{R}^{n}\setminus B_{R})$ becomes
\begin{equation} \label{equricciflat}
   \triangle_{g} g_{ij}= Q_{ij}(g,\partial g).
\end{equation}
The readers can find more knowledge about harmonic coordinates and elliptic estimate theory in \cite{BKN} and \cite{Ba}.

\section{The Sets $\mathcal{T}$ and $\widetilde{\mathcal{T}}$}\label{sec:set}

In this section, we define and study the function spaces $\mathcal T$ and $\widetilde{\mathcal T}$ step by step. They are used to formulate the expansion in our main theorem and indeed, the regularity property of $g_{ij}$ is encoded in their definition.

\subsection{Expansion of bounded harmonic functions}
Now let us start from the expansion of bounded harmonic functions on $\mathbb{R}^{n}$.

Letting $(r,\omega)$ be the polar coordinates, we compute
   \begin{equation}\label{polarharmonic}
   \triangle_{\mathbb{R}^{n}}H(x)=r^{-(n-1)}\partial_{r}(r^{(n-1)}\partial_{r}H)
  +r^{-2}\triangle_{S^{n-1}}H.
   \end{equation}
  If $H(x)=r^{\sigma}G(\omega)$ for some $\sigma\in \mathbb{R}$, then
  \begin{equation*}
    \triangle_{\mathbb{R}^{n}}H(x)=r^{\sigma-2}(\triangle_{S^{n-1}}G+\sigma(\sigma+n-2)G).
    \end{equation*}
    So $H$ is harmonic on $\mathbb{R}^{n}\setminus \{0\}$ if and  only if $G$ is an eigenfunction of $\triangle_{S^{n-1}}$ with eigenvalue $-\sigma(\sigma+n-2)$. Since the eigenvalue is nonpositive, the possible values of $\sigma$ are $0,1,\cdots$ and $2-n,1-n,-n,\cdots$. Boundedness of $H$ suggests the following lemma, whose proof is well-known and can be found
    in \cite{bookharmonicfunction}.

  \begin{lem}\label{lem:harmonic expansion}
	Let $H$ be a bounded harmonic function on $\mathbb{R}^{n}\setminus B_{2}$, then there exist constants $c_k$ and $c$ such that
	\begin{equation*}
		H(x)=\sum^{\infty}_{k=0}c_{k}r^{-k-(n-2)}G_{k} +c,
	\end{equation*}
	where $G_{k}$ is an eigenfunction of $\triangle_{S^{n-1}}$ corresponding to the eigenvalue $-k(k+n-2)$.
\end{lem}

   Motivated by this lemma, we define
\begin{equation*}
\mathcal{T}_{1}:= \{r^{-n-k+2}G_{k}; \quad k\in \mathbb{N}, G_{k}\in \mathcal{H}_{k}(S^{n-1})\}.
\end{equation*}
Recall that $\mathcal H_k(S^{n-1})$ is the space of eigenfunctions corresponding to the eigenvalue $-k(k+n-2)$. Notice that $\mathcal T_1$ is almost a basis of bounded harmonic function and we have removed the constant function and the function $r^{2-n}$ from it, because these are prepared for the expansion of $g_{ij}-\delta_{ij}$, which know (from the work of Bando, Kasue and Nakajima) that it decays at order $n-1$.

\subsection{Closedness under multiplication and differentiation}
For the proof in this subsection, we need to introduce more notations and facts concerning the spherical harmonics.
  Let $\mathcal{P}_{k}(n)$ denote the space of homogeneous polynomials of degree $k$ in $n$ variables with real coefficients  and let
$\mathcal{P}_{k}(S^{n-1})$  denote the restrictions of homogeneous polynomials in $\mathcal{P}_{k}(n)$ to
$S^{n-1}$.
Set $\mathcal{H}_{k}(n)$
to be the space of real harmonic homogeneous polynomials of degree $k$. We have a
linear isomorphism between $\mathcal{H}_{k}(n)$ and $\mathcal{H}_{k}(S^{n-1})$ by restricting functions on $\mathbb{R}^{n}$ to
$S^{n-1}$.

	Now we can define
 $$\mathcal{T}_{2}:=\{r^{(2-n)l-k}G_{m};\quad k\geq l\geq 1, k, l,m ,
 \frac{k-m}{2}\in \mathbb{N}\cup\{0\},G_{m}\in \mathcal{H}_{m}(S^{n-1})\}.$$
\begin{rem}  It is possible that $r^{(2-n)l_{1}-k_{1}}G_{m}=r^{(2-n)l_{2}-k_{2}}G_{m}$ but $(l_1,k_1)\ne (l_2,k_2)$. It is in the set as long as one pair of $(l,k)$ satisfies the requirement. The same remark applies for similar definitions in what follows.
 \end{rem}

 Please note that we have tried to make sure that (1) it satisfies the properties below and (2) it is as small as possible, because the smaller the set is, the better regularity we prove.
\begin{lem}\label{the2thset}
 $\mathcal T_2$ is a set with the following properties:

  1) $\mathcal{T}_{1}\subset\mathcal{T}_{2}$,

  2) $\mbox{SPAN}(\mathcal{T}_{2})$ is closed under multiplication,

  3) for any $u\in \mathcal T_2$,
  $$\pfrac{u}{x_i} \in \mbox{SPAN}(\mathcal{T}_{2}), \qquad i=1,2,\cdots,n.$$
\end{lem}

 \begin{proof}
 First $\mathcal{T}_{1}\subset\mathcal{T}_{2}$ is obvious. For 2), if
  $$X_{1}=r^{(2-n)l_{1}-k_{1}}G_{m_{1}},X_{2}=r^{(2-n)l_{2}-k_{2}}G_{m_{2}}\in\mathcal{T}_{2},$$
  then
  \begin{equation}\label{product2}
  X_{1}\cdot X_{2}=r^{(2-n)(l_{1}+l_{2})-(k_{1}+k_{2})}G_{m_{1}}\cdot G_{m_{2}}.
  \end{equation}
  Because $r^{m}G_{m}$ is a harmonic polynomial,
  \begin{equation}\label{harprod}
  r^{m_{1}}G_{m_{1}}r^{m_{2}}G_{m_{2}}\in\mathcal{P}_{m_{1}+m_{2}}(n).
  \end{equation}
 Recalling (see \cite{bookharmonicfunction}) that
 \begin{equation}\label{eqn:decomposition}
 \mathcal{P}_{k}(n) =
\mathcal{H}_{k}(n) \oplus |x|^{2}\mathcal{H}_{k-2}(n)
\oplus \cdots \oplus
|x|^{2j}\mathcal{H}_{k-2j}(n)
\oplus  \cdots \oplus
 |x|^{2[\frac{k}{2}]}\mathcal{H}_{k-2[\frac{k}{2}]}(n),
 \end{equation}
 and then by restricting (\ref{harprod}) to $S^{n-1}$, we have
 \begin{equation}\label{eqn:decomposition*}
 G_{l}\cdot G_{m}\in\mathcal{H}_{l+m}(S^{n-1})\oplus \cdots\oplus
 \mathcal{H}_{l+m-2j}(S^{n-1})\oplus \cdots\oplus
 \mathcal{H}_{l+m-2[\frac{l+m}{2}]}(S^{n-1}).
 \end{equation}
Therefore (\ref{product2}) and (\ref{eqn:decomposition*}) imply  $X_{1}\cdot X_{2}\in\text{SPAN}(\mathcal{T}_{2}).$

For the proof of 3), using the fact that
	\begin{equation}\label{smallfact}
	\frac{\partial r}{\partial x^{i}}=\frac{x^{i}}{r}\in\mathcal{H}_{1}(S^{n-1}),
\end{equation}
we compute
\begin{eqnarray*}
	& &\partial_{x_i} (r^{(2-n)l-k}G_{m})= \partial_{x_i}\left( r^{(2-n)l-k-m}r^{m}G_{m}\right) \\
    &=&r^{(2-n)l-k-m}(\partial_{x^{i}} (r^{m}G_{m}))+((2-n)l-k-m)r^{(2-n)l-k-1}\cdot \frac{x^{i}}{r}\cdot G_{m}\\
    &=&r^{(2-n)l-k-1}(G_{m-1}+G_{m-3}+\cdot\cdot\cdot)+r^{(2-n)l-k-1}(G_{m+1}+G_{m-1}+\cdot\cdot\cdot)\\
    &=&r^{-1-k-(n-2)l}
 (G_{m+1}+G_{m-1}+\cdots+G_{m+1-2[\frac{m+1}{2}]}).
\end{eqnarray*}
Here in the last equality we have used (\ref{eqn:decomposition}), (\ref{eqn:decomposition*}) and
the fact $r^{m}G_{m}$ is a homogenous harmonic polynomial. Consequently,
 \begin{eqnarray*}
	 \partial_{x_i} (r^{-k-(n-2)l}G_{m})\in \mbox{SPAN}(\mathcal{T}_{2}),
 \end{eqnarray*}
if
\begin{eqnarray*}
	 r^{-k-(n-2)l}G_{m}\in \mbox{SPAN}(\mathcal{T}_{2}).
 \end{eqnarray*}
 This finishes the proof of the lemma and explains the definition of $\mathcal T_2$.
\end{proof}

\begin{rem}
In the above proof, the function $G_m\in\mathcal{H}_{m}(S^{n-1})$ may vary from line to line. And in the rest of paper, we won't
specially point it out when this happens.
\end{rem}

\subsection{Almost closedness under the inverse of $\triangle$}\label{sec:inverse}

For the proof of our main theorem, we need to consider the operator $\triangle^{-1}$. Very roughly speaking, the effect of $\triangle$ is just reducing the power of $r$ by $2$. More precisely, for almost all $\sigma\in \mathbb{R}$, $\triangle r^\sigma G_m$ is a nonzero multiple of $r^{\sigma-2}G_m$. There are exceptional cases, namely, if $r^\sigma G_m$ is harmonic, we do not expect the same result.

As it turns out (see later proofs), for $n>5$, no exceptional case happens. So we simply define
 \begin{eqnarray*}
	{\mathcal T} &:=&\Big\{r^\sigma G_m; \quad G_m\in \mathcal{H}_{m}(S^{n-1}), \sigma=2j -(n-2)l-k,\\
                 & &\quad k\geq l \geq j+1,\quad l,k,m,j,\frac{k-m}{2}\in \mathbb N\cup \{0\}\Big\}.
\end{eqnarray*}
We can check by using the same proof as Lemma \ref{the2thset} (details are omitted),
\begin{lem} \label{finalset5}
 $\mathcal T$ is a set satisfying the three properties below:

  1) $\mathcal{T}_{2}\subset\mathcal{T}$,

  2) $ \mbox{SPAN}(\mathcal T)$ is closed under multiplication,

  3) for any $u\in \mathcal T$,
  $$\frac{\partial u}{\partial x_i} \in \mbox{SPAN}(\mathcal{T}), \qquad i=1,2,\cdots,n.$$
\end{lem}

However, when $n=4$, the exceptional case does happen and we are forced to introduce $\log$ terms in the expansion. The problem is how many $\log$ terms are needed, because we certainly do not want anything more than necessary. The answer lies in the proof in Section \ref{sec:str} and here we can only define
	\begin{equation}\label{4set}
		\begin{split}
		\widetilde{\mathcal{T}} &:=\Big\{r^\sigma (\log r)^{i}G_m; \quad  G_m\in \mathcal{H}_{m}(S^{3}),\sigma=2j -(4-2)l-k, \\
                 & \quad k\geq l \geq j+1,j\geq i,\quad l,k,i,m,j,\frac{k-m}{2}\in \mathbb N\cup \{0\}\Big\}\\
                 & =\Big\{r^\sigma (\log r)^{i}G_m; \quad G_m\in \mathcal{H}_{m}(S^{3}), \sigma= -2l-k,\\
                 &  \quad l\geq 1,k\geq l+i, \quad l,k,i,m ,\frac{k-m}{2} \in\mathbb N\cup \{0\}\Big\}	
		\end{split}
		\end{equation}
and check some basic properties.
\begin{lem}\label{lem:4dimset}
	
	$\widetilde{\mathcal T}$ is a set satisfying the three properties below:

  1) $\mathcal{T}_{2}\subset \widetilde{\mathcal{T}}$,

  2) $ \mbox{SPAN}(\widetilde{\mathcal{T}})$ is closed under multiplication,

  3) for any $u\in \widetilde{\mathcal T}$,
  $$\frac{\partial u}{\partial x_i} \in \mbox{SPAN}(\widetilde{\mathcal{T}}), \qquad i=1,2,\cdots,n.$$
\end{lem}

\begin{proof}
1) and 2) are easy to see by the argument used in the proof of Lemma \ref{the2thset}. We only prove 3). Set $\sigma=-2l-k$,
by using (\ref{smallfact}) we obtain
\begin{eqnarray*}
	& &\frac{\partial}{\partial x^{k}}(r^\sigma (\log r)^{i}G_m)=\frac{\partial}{\partial x^{k}}(r^{\sigma-m} (\log r)^{i}r^{m}G_{m})\\
    &=&r^{\sigma-m} (\log r)^{i}(\partial x^{k}(r^{m}G_{m}))+i\cdot r^{\sigma-m} (\log r)^{i-1}r^{-1}\cdot \frac{x^{k}}{r}\cdot  r^{m}G_{m}\\
    & &+(\sigma-m)r^{\sigma-m-1} (\log r)^{i}\cdot \frac{x^{k}}{r}\cdot r^{m}G_{m}\\
    &=&r^{\sigma-1} (\log r)^{i}(G_{m-1}+G_{m-3}+\cdot\cdot\cdot)+i\cdot r^{\sigma-1} (\log r)^{i-1}(G_{m+1}+G_{m-1}+\cdot\cdot\cdot)\\
    & &+(\sigma-m)r^{\sigma-1} (\log r)^{i}(G_{m+1}+G_{m-1}+\cdot\cdot\cdot).
\end{eqnarray*}
Here again in the last equality we have used (\ref{eqn:decomposition}), (\ref{eqn:decomposition*}) and
the fact $r^{m}G_{m}$ is a homogenous harmonic polynomial. It's straightforward to check that every term in the above formula
can be written in the form
\begin{eqnarray*}
& &r^{\sigma-1} (\log r)^{\alpha}G_{\beta}:=f_{\alpha,\beta}, \\
& &\alpha=i,i-1,\quad\beta=m+1, m-1,\cdots,{m+1-2[\frac{m+1}{2}]}.
\end{eqnarray*}
Let $ \tilde{k}=k+1,$ then $ \sigma-1=-2l-\tilde{k}$.
At once we have
$$ f_{\alpha,\beta}=r^{-2l-\tilde{k}} (\log r)^{\alpha}G_{\beta} \in\widetilde{\mathcal{T}},$$
  if $r^\sigma (\log r)^{i}G_m $ belongs to
$ \widetilde{\mathcal{T}}$, since $l\geq 1,\tilde{k}\geq l+\alpha+1,$ and $l,\tilde{k},\alpha,\beta ,\frac{\tilde{k}-\beta}{2} \in\mathbb N\cup \{0\}$ (see(\ref{4set})).
Therefore 3) holds.
\end{proof}

\subsection{The inverse of the metric matrix}
In this subsection, we prove a lemma for later use. Our aim is to show that $g^{ij}-\delta_{ij}$ has some expansion if $g_{ij}-\delta_{ij}$ is assumed to have an expansion up to order $-q$.

We need the following result, which is similar to Lemma 6.8 in \cite{yin2016analysis}.
 \begin{lem}\label{composition}
 For any smooth function $F:\mathbb{R}\rightarrow \mathbb{R}$, if $f\in \mathcal{C}^{\infty}_{1-n}$ has an expansion up to order $-q$ in the sense of Definition \ref{defexpansionintr} with terms in $A$, which is a set meets the condition $ \text{SPAN}(A)$ is closed under multiplication, then so do $F\circ f-F(0)$.
\end{lem}
For the reader's convenience we give the proof here.
Before the proof, we list several obvious equalities, which can be checked straightforwardly:

\begin{equation}\label{ordercomputatation}
	\begin{split}
& r^\sigma G_m O(-q)=O(\sigma-q),\\
& \partial_{x_i}O(-q)=O(-(q+1)),\\
& O(-q_{1})O(-q_{2})=O(-q_{1}-q_{2}),
	\end{split}
\end{equation}
where $q,q_{1},q_{2}\geq0$.

\begin{rem}\label{productexp}
We claim that if $f_{1},f_{2}\in\mathcal{C}^{\infty}_{1-n}$
both have an expansion up to order $-q$ in terms of $A$ in the sense of Definition \ref{defexpansionintr}, then
		so do $f_{1}+f_{2}$	and $f_{1}\cdot f_{2}$.
The claim holds trivially for $f_{1}+f_{2}$. For $f_{1}\cdot f_{2}$, it follows from the assumption that $ \text{SPAN}(A)$ is closed under multiplication and the first and third equality in (\ref{ordercomputatation}).
\end{rem}

\begin{proof}
 We assume that for some $ \epsilon>0$ there exists  $ f_{q}\in \text{SPAN}(A) $ such that $f=f_{q}+O(-q-\epsilon)$.
Recall the Taylor expansion formula with the integral remainder,
\begin{equation*}
F(x)=\sum_{l=0}^{k}\frac{F^{(l)}(0)}{l!}x^{l}+\frac{1}{k!}
\Big(\int_{0}^{1}F^{(k+1)}(tx)(1-t)^{n}\,dt\Big)x^{k+1},
\end{equation*}
in which we choose $k$ so that $(k+1)(1-n)<-q$.
By the assumption $ \text{SPAN}(A)$ is closed under multiplication,
\begin{equation*}
\sum_{l=0}^{k}\frac{F^{(l)}(0)}{l!}(f)^{l}-F(0)
\end{equation*}
has an expansion up to order $-q$, due to the claim in Remark \ref{productexp}. To be more precise, it takes the form $F_{q}+O(-q-\epsilon)$, $ F_{q}\in \text{SPAN}(A) $.
Thanks to our choice of $k$ and the fact $f\in \mathcal{C}^{\infty}_{1-n}$, which implies $f=O(1-n)$, we have
$ (f)^{k+1}=O(-q-\epsilon)$ for some $ \epsilon>0$.
By direct computation, one can check $$\int_{0}^{1}F^{(k+1)}(tf)(1-t)^{n}\,dt=O(0),$$ therefore $$\Big(\int_{0}^{1}F^{(k+1)}(tf)(1-t)^{n}\,dt\Big)(f)^{k+1}=O(-q-\epsilon).$$
Thus we can conclude $F\circ f-F(0)=F_{q}+O(-q-\epsilon)$, $ F_{q}\in \text{SPAN}(A) $.
\end{proof}

The following lemma roughly states that a set which is closed under multiplication (e.g., $\text{SPAN}(\mathcal{T})$,
$ \text{SPAN}(\widetilde{\mathcal{T}}$)) is closed under the matrix inversion.
 \begin{lem}\label{leminversemetric}
 Let $A$ be a set which has the property
 $ \text{SPAN}(A)$ is closed under multiplication.
If $(g_{ij}-\delta_{ij})\in \mathcal{C}^{\infty}_{1-n}$ has an expansion up to order $-q$ in terms of $A$ in the sense of Definition \ref{defexpansionintr}, then $g^{ij}-\delta_{ij}$ has an expansion up to order $-q$ in the same sense.
\end{lem}

\begin{proof}

If $g_{ij}-\delta_{ij}$ has a expansion up to order $-q$ in $A$ in the sense of Definition \ref{defexpansionintr}, then formally
\begin{equation} \label{matrix11}
g_{ij}-\delta_{ij}=U_{q,ij}+O(-q-\epsilon),
\end{equation}
that is $$[g_{ij}]_{n\times n}=[\delta_{ij}]_{n\times n}+[U_{q,ij}+O(-q-\epsilon)]_{n\times n},$$
where $U_{q,ij}\in \text{SPAN}(A)$.

The inverse matrix $ [g^{ij}]$ is computed by
\begin{equation}\label{inverformular}
g^{ij}=\frac{A^{ij}}{\det(g)},
\end{equation}
 where $[A^{ij}]$ is the adjoint matrix of $[g_{ij}]$.
By the knowledge of linear algebra
 \begin{eqnarray*}
 A^{ij}&=&(-1)^{i+j}\det([A(j,i)]).
 \end{eqnarray*}
 The matrix $[A(j,i)] $ which is obtained by deleting the $j$th row and the $i$th column of $g$ is called the cofactor of $g$.

 When $i=j$, since $[A(i,i)]$ has the form of $[\delta_{kl}+c_{kl}]_{(n-1)\times(n-1)}$, using the equality
  \begin{equation}\label{det}
  \begin{split}
 \det ([\delta_{ij}+c_{ij}])
      =&1+\sum_{k=1}^{n}c_{kk}+\frac{1}{2!}\sum_{k_{1},k_{2}}\det\left([\begin{array}{cc}
      c_{k_{1}k_{2}} & c_{k_{1}k_{2}}\\ c_{k_{2}k_{1}} & c_{k_{2}k_{2}} \end{array}]\right)\\
       &+\cdots+\det([c_{ij}]),
      \end{split}
\end{equation}
we have
\begin{equation}\label{inversepart1}
  \begin{split}
 A^{ii}=&1+\sum_{k=1,k\neq i}^{n}U_{q,kk}+\cdots +O(-q-\epsilon)\\
 :=&1+W_{q,ii}+O(-q-\epsilon).
      \end{split}
\end{equation}
By virtue of the claim in Remark \ref{productexp}, (\ref{matrix11}) and the assumption on $A$, we have $W_{q,ii}\in \text{SPAN}(A)$.

Because the computations are the same when we calculate $A^{ij}$ no matter $i<j$ or $ i>j$, we just deal with the former case.
When $i<j$, we first point out two facts so as to calculate $A^{ij}$.
They are
\begin{itemize}
		\item from the calculating methods of determinant, $\det([A(i,j)])$ is the sum of terms which are of the form
\begin{equation}\label{cofactor1}
g_{1i_{1}}\;\cdots \;g_{j-1i_{j-1}}\;g_{j+1i_{j+1}}\;\cdots \;g_{ni_{n}},
\end{equation}
where $i_{1}\cdots i_{j-1}i_{j+1}\cdots i_{n} $ is a permutation of $\set{1,\cdots,n}\setminus \set{i} ;$
		\item the only coefficients of $[A(i,j)]$ with constant term are the following $n-2$ coefficients
$$ g_{kk},\,k\in \set{1,\cdots,n}\setminus \set{i,j}.$$
	\end{itemize}
From the above two facts, we know there is no constant term in the expansion of $\det([A(i,j)])$, since at least one factor of
(\ref{cofactor1}) has no constant term.
Again by the claim in Remark \ref{productexp}, (\ref{matrix11}) and the assumption on $A$, we can conclude $A^{ij}$ has the form
\begin{eqnarray}\label{inversepart12}
 A^{ij}=W_{q,ij}+O(-q-\epsilon),
 \end{eqnarray}
where $W_{q,ij}\in \text{SPAN}(A)$.

By using (\ref{det}) again,
\begin{eqnarray*}
\det(g)&=&1+\sum_{k=1}^{n}U_{q,kk}+\cdots+\det([{U_{q,ij}}])+O(-q-\epsilon)\\
      &:=&1+V_{q}+O(-q-\epsilon).
\end{eqnarray*}
Once again, by the claim in Remark \ref{productexp}, (\ref{matrix11}) and the assumption on $A$, $V_{q}\in \text{SPAN}(A)$.

If we choose $F=\frac{1}{1+x}$ and $f=V_{q}+O(-q-\epsilon)$ while using Lemma \ref{composition},
we obtain
\begin{equation}\label{inversepart3}
(\det(g))^{-1}=1+Z_{q}+O(-q-\epsilon),
\end{equation}
where $Z_{q}\in \text{SPAN}(A)$.

Putting (\ref{inversepart1}), (\ref{inversepart12}) and (\ref{inversepart3}) together via (\ref{inverformular}) we have
\begin{eqnarray*}\label{metricinverse}
g^{ij}&=&(\delta_{ij}+W_{q,ij}+O(-q-\epsilon))(1+Z_{q}+O(-q-\epsilon))\nonumber\\
&:=&\delta^{ij}-\tilde{U}_{q,ij}+O(-q-\epsilon),
\end{eqnarray*}
where $\tilde{U}_{q,ij}\in \text{SPAN}(A)$ by the claim in Remark \ref{productexp}.
Hence $g^{ij}-\delta_{ij} $ has an expansion in the sense of Definition \ref{defexpansionintr}.

\end{proof}

\begin{rem}\label{struorder}
If $A$ is $\mathcal{T}$ or $ \widetilde{\mathcal{T}}$,
then under the conditions of the last lemma we have from the definition of
$\mathcal{T}$ and $ \widetilde{\mathcal{T}}$ that
\begin{equation*}\label{ordernoconst}
g_{ij},g^{ij}= \delta_{ij}+O(1-n).
\end{equation*}
\end{rem}

\section{The Solutions Of Poisson Equations}\label{sec:sol}

As we have mentioned in the introduction, we need to solve the Poisson equations on $ \mathbb{R}^{n}\setminus B_{R}$ in the bootstrapping argument. This section is devoted to the discussion of this topic. It consists of two parts. In the first part, we study carefully the action of $\triangle^{-1}$ on $\mbox{SPAN}(\mathcal T)$ and $\mbox{SPAN}(\widetilde{\mathcal T})$. In the second part, we introduce a lemma due to Meyers \cite{Mey}, which handles the error term.

\subsection{Functions in $\mbox{SPAN}(\mathcal T)$ and $\mbox{SPAN}(\widetilde{\mathcal T})$ }
The discussion in this subsection is completely elementary. A term in $\mathcal T$ or $\widetilde{\mathcal T}$ is of the form
\begin{equation*}
	r^\sigma (\log r)^i G_m,
\end{equation*}
where $\sigma$, $i$ and $m$ are subject to some restrictions in the definition of $\mathcal T$ or $\widetilde{\mathcal T}$.

We are interested in finding a solution of the equation
\begin{equation}\label{poissonsingletermlog}
\triangle_{\mathbb{R}^{n}}u=r^{\sigma}(\log r)^{i}G_{m}
\end{equation}
on $ \mathbb{R}^{n}\setminus B_{1}$.
The following Proposition indicates that there is a subtle difference depending on whether or not $\sigma+n+m$ is zero. Hence, we call the right hand side {\bf exceptional} if
\begin{equation}
	\sigma+n+m=0.
	\label{exceptionaldef}
\end{equation}
%

\begin{prop}\label{solutionlog}
If $\sigma+n+m\neq 0$, then it has a solution taking the form
\begin{equation}\label{solutionlog1}
r^{\sigma+2}(\sum_{j=0}^{i}c_{j}(\log r)^{j})G_{m},c_{j}\in \mathbb{R};
\end{equation}
if $\sigma+n+m=0$, then it has a solution taking the form
\begin{equation}\label{solutionlog2}
r^{\sigma+2}(\sum_{j=0}^{i+1}c_{j}(\log r)^{j})G_{m},c_{j}\in \mathbb{R}.
\end{equation}
\end{prop}
\begin{rem}
	In later proofs, when $n>4$, we shall only use the first case with $i=0$. For $n=4$, both cases are needed.
\end{rem}

\begin{proof}
For simplicity of notations, we set $\lambda_{k}:=k(k+n-2)$ for all $k\in\mathbb{Z}\setminus \set{-1,\cdots,3-n}$ and
$n\ast \sigma:= n+2\sigma+2$.
We first prove the proposition when $\sigma+n+m\neq 0$.
By (\ref{polarharmonic}),
\begin{equation}\label{calculatelaplace}
	\begin{split}
& \triangle_{\mathbb{R}^{n}}(r^{\sigma+2}(\log r)^{j}G_{m})\\
=&(\lambda_{\sigma+2}-\lambda_{m})(r^{\sigma}(\log r)^{j}G_{m}) +j(n\ast \sigma)(r^{\sigma}(\log r)^{j-1}G_{m})\\
& +j(j-1)(r^{\sigma}(\log r)^{j-2}G_{m}).
	\end{split}
\end{equation}
Let $U_{j}=r^{\sigma+2}(\log r)^{j}G_{m}$ for $j=0,1,\cdots,i$, (\ref{calculatelaplace}) shows that
\begin{equation*}
\triangle_{\mathbb{R}^{n}}[U_{0},\cdots, U_{j},\cdots,U_{i}]=r^{-2}[U_{0},\cdots, U_{j},\cdots,U_{i}][A_{0},\cdots,A_{j},\cdots,A_{i}].
\end{equation*}
Here $[U_0,\cdots,U_i]$ is a row vector and each $A_j$ is a column vector, which means that $[A_0,\cdots,A_i]$ is an $(i+1)\times (i+1)$ matrix. By \eqref{calculatelaplace}, we know

(1) $A_{0}=[(\lambda_{\sigma+2}-\lambda_{m}),0,\cdots,0]^{\top} $;

(2) $A_{1}=[n\ast \sigma,(\lambda_{\sigma+2}-\lambda_{m}),0,\cdots,0]^{\top}$;

(3) for $2\leq j\leq i$, $A_{j}=[0,\cdots,0,j(j-1),j(n\ast \sigma),(\lambda_{\sigma+2}-\lambda_{m}),0,\cdots,0]^{\top} $ and the three nonzero numbers are at
$j-1$th, $j$th, and $j+1$th positions.

Since $\sigma+n+m\neq 0$, which implies $\lambda_{\sigma+2}-\lambda_{m}:=\lambda_{\sigma,m} \neq 0$, we have that the matrix
\begin{eqnarray*}
A&:=&[A_{0},\cdots,A_{j},\cdots,A_{i}]_{(i+1)\times(i+1)}\\
&=&\left[  \begin{array}{ccccccc}\lambda_{\sigma,m} & (n\ast \sigma) & 2 & 0&0& \cdots & 0 \\
0 & \lambda_{\sigma,m} & 2(n\ast \sigma)& 6 & 0 &\cdots &0\\
0 & 0 & \lambda_{\sigma,m}& 3(n\ast \sigma) & 12 &\cdots &0\\
\multicolumn{7}{c}\dotfill\\
\multicolumn{7}{c}\dotfill\\
\multicolumn{7}{c}\dotfill\\
0 & \cdots&\cdots& 0& \lambda_{\sigma,m}&(i-1)(n\ast \sigma)&i(i-1)\\
0&\cdots&\cdots&\cdots&0&\lambda_{\sigma,m}&i(n\ast \sigma)\\
0&\cdots&\cdots&\cdots&\cdots&0&\lambda_{\sigma,m}
\end{array}\right]
\end{eqnarray*}
is invertible.
Hence (\ref{poissonsingletermlog}) has a solution of the type of (\ref{solutionlog1}).

When $\sigma+n+m=0$, we need to enlarge the solution space to include
\begin{equation*}
U_{i+1}=r^{\sigma+2}(\log r)^{i+1}G_{m}.
\end{equation*}
Using (\ref{calculatelaplace}) as before, we get
  \begin{equation*}
\triangle_{\mathbb{R}^{n}}[U_{0},\cdots, U_{j},\cdots,U_{i+1}]=r^{-2}[U_{0},\cdots, U_{j},\cdots,U_{i+1}][A_{0},\cdots,A_{j},\cdots,A_{i+1}].
\end{equation*}
Since now we have $\sigma+ n+m=0$, which implies $\lambda_{\sigma+2}-\lambda_{m} =0$, all the diagonal elements in the $(i+2)\times(i+2)$ upper triangular matrix below
\begin{eqnarray*}
\bar{A}&:=&[A_{0},\cdots,A_{j},\cdots,A_{i+1}]_{(i+2)\times(i+2)}\\
&=&\left[  \begin{array}{ccccccc}0 & (n\ast \sigma) & 2 & 0&0& \cdots & 0 \\
0 & 0 & 2(n\ast \sigma)& 6 & 0 &\cdots &0\\
0 & 0 & 0& 3(n\ast \sigma) & 12 &\cdots &0\\
\multicolumn{7}{c}\dotfill\\
\multicolumn{7}{c}\dotfill\\
\multicolumn{7}{c}\dotfill\\
0 & \cdots&\cdots& 0& 0&i(n\ast \sigma)&i(i+1)\\
0&\cdots&\cdots&\cdots&0&0&(i+1)(n\ast \sigma)\\
0&\cdots&\cdots&\cdots&\cdots&0&0
\end{array}\right]
\end{eqnarray*}
are zeros, so it is not invertible.
 And it is easy to see its rank is $i+1$ since
 \begin{equation*}
 n\ast \sigma=n+2\sigma+2=-n-2m+2<0
  \end{equation*}
 by $n\geq4$ and $m\geq0$.
  Moreover because the last row of $\bar{A}$ is a zero matrix, we conclude that the map
\begin{equation*}
 \triangle_{\mathbb{R}^{n}}:\quad \text{SPAN}(\set{U_{0},\cdots, U_{j},\cdots,U_{i+1}})\rightarrow \text{SPAN}(\set{r^{-2}U_{0},\cdots, r^{-2}U_{j},\cdots,r^{-2}U_{i}})
 \end{equation*}
 is surjective.
 Therefore (\ref{poissonsingletermlog}) has a solution of the type of (\ref{solutionlog2}).
\end{proof}

\subsection{The error terms}

The following is a result about the error estimates of Poisson equations.
\begin{lem}\label{corpoissonestimate4}
	Let $G$ be a function defined on $\Real^n\setminus B_R (n>2)$ which is $G=O(-2-p-\epsilon)$ for some $p\in\mathbb{N}$, $p>0$ and $1>\epsilon>0$, then there is a solution $V^{\star}$ of
\begin{equation}\label{equ:poissoonestimate2}
\triangle V^{\star} =G,
\end{equation}
and
	\begin{equation*}
		V^{\star} = O(-p-\epsilon).
	\end{equation*}
\end{lem}

This lemma follows directly from the results of \cite{Mey} which is devoted to the study of the expansion for
 solutions of linear elliptic equations.

Next we introduce some results in \cite{Mey} which are essential to the proof of Lemma \ref{corpoissonestimate4}.
Let $ \mu=r^{-p}$ for some $p>0$. Now for any function $U$ which is $i$ times differentiable and
$|x|\geq R\geq2$, define
\begin{eqnarray*}
	& &H^{i}_{\alpha}[x,U]=\sup_{\{y;|x-y|<\frac{1}{2}|x|\}} \frac{D^{i}U(x)-D^{i}U(y)}{|x-y|^{\alpha}}
   \,\,\,(0<\alpha<1),\\
& &|U|^{R}_{\mu,i}=\sum^{i}_{j=0}\sup_{|x|\geq R}\frac{1}{\mu(|x|)}|x|^{i}|D^{i}U(x)|,\\
& &	|U|^{R}_{\mu,i+\alpha}=|U|^{R}_{\mu,i}+\sup_{|x|\geq R}\frac{1}{\mu(|x|)}|x|^{i+\alpha}H^{i}_{\alpha}[x,U].
\end{eqnarray*}

It is not hard to check that if $f=O(-p),$
then $|f|^{R}_{r^{-p},i+\alpha}=O(1)$ for all $i\in\mathbb{N}\cup\set{0}$ as $R\rightarrow \infty$.
It is noted here that $ O(1)$ takes the usual meaning in standard mathematical analysis textbooks.
Indeed if $f=O(-p),$ then for all $i\in\mathbb{N}\cup\set{0}$,
we have $|r^{j}D^{j}f|\leq C(j)r^{-p}$ as $x\rightarrow \infty$ for all $0\leq j\leq i+1$. So there exists some $R_0>0$ such that if
$R>R_0$, then
$$\sup_{|x|\geq R}|x|^{j}|D^{j}f(x)|\leq C(j)r^{-p}\; \text{for all}\; 0\leq j\leq i+1,$$
which implies
\begin{eqnarray*}
	|f|^{R}_{\mu,i+1}=\sum^{i+1}_{j=0}\sup_{|x|\geq R}\frac{1}{r^{-p}}|x|^{j}|D^{j}f(x)|\leq C(i)
\end{eqnarray*}
for some $C(i)>0$.
The last inequality above clearly shows that $|f|^{R}_{r^{-p},i+\alpha}=O(1)$ for all $i\in\mathbb{N}\cup\set{0}$ as $R\rightarrow \infty$.

\begin{lem}(Theorem 1. and Corollary 1. in \cite{Mey})
\label{corderivativestimates}
Let $U$ be a solution of
$\triangle U=f$ on $\mathbb{R}^{n}\setminus B_{1}$, if
\begin{eqnarray*}
U=O(\mu(|x|))\,\, and\,\,|f|^{R}_{r^{-2}\mu,i+\alpha}=O(1)\,\, as\,\,R\rightarrow \infty,
\end{eqnarray*}
then
\begin{eqnarray*}
D^{j}U=O(\mu(|x|)|x|^{-j}),\,\,\forall j\leq i+2.
\end{eqnarray*}
\end{lem}

\begin{lem}\label{equ:poissoonestimate}
( Lemma 5 in \cite{Mey}) Consider the equation
$$\triangle V=G=O(r^{-2-p-\epsilon}(\log r)^{q})$$ on $\mathbb{R}^{n}\setminus B_{1}$, and G
 is in $\mathcal{C}^{\alpha}$, $ p\in\mathbb{Z},q\in\mathbb{N}\cup \set{0}$ and $0\leq\epsilon<1$. For $n>2$, the equation has a solution
 \begin{eqnarray*}
 V^{\star}(x)=\left\{\begin{array}{cc}O(r^{-p-\epsilon}(\log r)^{q})&\quad\text{if}\,\, 0<p<n-2\,\, \text{or}\,\, 0<\epsilon<1,\\
O(r^{-p}(\log r)^{q+1})&\quad \text{otherwise}.
\end{array} \right.
\end{eqnarray*}
\end{lem}

We remark here that we only use the case $q=0$ and $\epsilon>0$ in Lemma \ref{equ:poissoonestimate}.
Now we give the proof of Lemma \ref{corpoissonestimate4}.

\begin{proof}
Since by the assumption of Lemma \ref{corpoissonestimate4}
  $G=O(-2-p-\epsilon)(\epsilon>0)$, then by Lemma \ref{equ:poissoonestimate} it turns out (\ref{equ:poissoonestimate2}) has a solution $$ V^{\star}(x)= O(r^{-p-\epsilon}) .$$

Since $G=O(-2-p-\epsilon)$,
 if we take $\mu=r^{-p-\epsilon}$, then  $$|G|^{R}_{r^{-2}\mu,i+\alpha}=O(1),$$
  as $R\rightarrow \infty$, for all $i\in\mathbb{N}$.
Finally by using Lemma \ref{corderivativestimates} we obtain the desired estimates,
 \begin{eqnarray*}
D^{j}V^{\star}(x)=O(r^{-p-j-\epsilon}),\quad \text{for all} \,\,j\in\mathbb{N}.
\end{eqnarray*}
Hence $V^{\star}(x)=O(-p-\epsilon)$.

\end{proof}

\section{Structure Analysis Of The Ricci Flat Equation}\label{sec:str}

\subsection{The starting point of the bootstrap argument}

First we show that every Ricci flat ALE metric under the condition of Theorem \ref{MAINTHM} has an expansion up to order $1-n$
in the sense Definition \ref{defexpansionintr}. This is the starting point of the bootstrapping argument.

From Theorem \ref{MAINTHM1}, we know $g_{ij}-\delta_{ij}\in \mathcal{C}^{\infty}_{1-n} $. Immediately, we have
\begin{equation}\label{stepzero}
 g_{ij}-\delta_{ij}=O(1-n).
\end{equation}

Now we refine the above asymptotical behavior.
First, by using (\ref{eqn:laplace oporater}) we rewrite (\ref{equricciflat}) to be
 \begin{eqnarray}\label{equricflatnew}
  &&\triangle_{\mathbb{R}^{n}}(g_{ij}-\delta_{ij}) \nonumber\\
 &=&-(g^{kl}-\delta^{kl})\partial_{kl}(g_{ij}-\delta_{ij})+Q(g,\partial g):=RHS_{ij},
 \end{eqnarray}
where
\begin{eqnarray} \label{eqn:nonliearterm}
& &Q(g,\partial g) \nonumber  \\
&=&-\sum_{k,l,p,q} g^{pq}g_{lj}\partial_{p} g_{ik}\partial_{q} g^{kl}
      -\sum_{k,l,p,q,r,s} 2g_{ik}g_{lj} g^{pq}g^{rs}\Gamma^{k}_{pr}\Gamma^{l}_{qs}.
\end{eqnarray}
Second, we plug (\ref{stepzero}) into the right hand side of (\ref{equricflatnew}), and by direct computation
using (\ref{ordercomputatation}) we get
\begin{equation*}
RHS_{ij}=O(-2n).
\end{equation*}
Third, we solve the equation
\begin{equation*}
\triangle_{\mathbb{R}^{n}}u=RHS_{ij}
\end{equation*}
on $ \mathbb{R}^n \setminus B_R$. We remark here that all the Poisson equations we are going to deal with later are formulated on $ \mathbb{R}^{n}\setminus B_{R}$.
By Lemma \ref{corpoissonestimate4}, it has a solution $ v^{l}_{ij}=O(l)$ for some $2-2n<l<3-2n$, since
$RHS_{ij}=O(-2n)$ implies $RHS_{ij}=O(l-2)$.

Therefore
\begin{equation}\label{firsterms}
g_{ij}-\delta_{ij}-v^{l}_{ij}:=H_{ij}
\end{equation}
is a harmonic function of order $r^{1-n}$.
Thus $H_{ij}\in SPAN(\mathcal{T}_{1})$ due to Lemma \ref{lem:harmonic expansion} and the definition of $\mathcal{T}_{1}$.

\begin{rem}\label{rem1thexpansion}
 If we write $H_{ij}=a_{1,ij} r^{-(n-1)}G_{1}+\tilde{H}_{ij}$ and set $w_{ij}=v^{l}_{ij}+\tilde{H}_{ij}$, where $\tilde{H}_{ij} $
 is of order $r^{-n}$, then
\begin{eqnarray*}\label{firsterrorterm}
g_{ij}-\delta_{ij}&=&a_{1,ij}r^{-(n-1)}G_{1}+w_{ij} ,\nonumber\\
D^{k}w_{ij}&=&O(r^{1-n-k-\epsilon})
\end{eqnarray*}
for all $k\in\mathbb{N}\cup \set{0}$ and some $\epsilon>0 .$
To put it another way, $g_{ij}-\delta_{ij}$ has an expansion with terms in $\mathcal{T}_{1}$ up to order $-(n-1)$ in the sense Definition \ref{defexpansionintr}. Notice that $\mathcal{T}_{1}\subset \mathcal{T}  (\widetilde{\mathcal{T} })$.
So we have the starting point for the bootstrapping.
\end{rem}

Actually we have from (\ref{firsterms}),
\begin{rem}\label{remaftermain}
 The expansion up to order $3-2n$ for
$g_{ij}-\delta_{ij} $
is of the form
 \begin{equation*}
g_{ij}-\delta_{ij}=\sum_{j=1}^{n-1}r^{-j-(n-2)}G_{j}+O(3-2n-\epsilon),
\end{equation*}
  for some $\epsilon>0 ,$ where $G_j\in \mathcal{H}_{j}(S^{n-1})$.
 That is to say the first $n-1$ terms in the expansion are all harmonic functions.
 \end{rem}

\subsection{The structure of the right hand side of the equation}\label{structure}

The purpose of this subsection is to study the $RHS$ of (\ref{equricflatnew}) under the assumption that $g_{ij}-\delta_{ij}$ has an expansion of order $-q_{N}:=N(1-n)$ ($N\in\mathbb{N}$) in the sense of Definition \ref{defexpansionintr}.

For the sake of simplicity of the following argument, if $A$ is a set of functions (e.g.; $\mathcal{T}$ or $\widetilde{\mathcal{T}}$),
we define $\mathcal{L}_{A\times A}:=\text{SPAN}(\set{uv; u,v \in A})$. Namely $\mathcal{L}_{A\times A}$ is the linear space generated by all functions of the form $uv$, where $u,v \in A$.

\begin{lem}\label{struresult}
Let $A$ be either $\mathcal T$ or $\widetilde{\mathcal T}$. Suppose that $g_{ij}-\delta_{ij}$ has an expansion in terms of $A$ up to order $-q_{N}$ in the sense of Definition \ref{defexpansionintr}. Then the $RHS$ in \eqref{equricflatnew} has an expansion in terms of $\mathcal{L}_{A\times A}$ up to order $-(q_{N}+n+1)$ in the same sense.
To put it in another way, $RHS_{ij}$ has the form
\begin{equation*}\label{righthand}
RHS_{q_{N},ij}+O(-(q_{N}+n+1)-\epsilon),
\end{equation*}
where $RHS_{q_{N},ij}$ is a finite linear combination in terms of the form
\begin{equation}\label{righthand1}
T_{1}\cdot T_{2} \quad \text{where} \;T_{i}\in A,i=1,2.
\end{equation}
\end{lem}

\begin{proof}
 By the assumption, there is $U_{q_{N},ij}$ in $\mbox{SPAN}(A)$ such that
\begin{equation*}
	g_{ij}-\delta_{ij}= U_{q_{N},ij} + O(-q_{N}-\epsilon)
\end{equation*}
for some $\epsilon>0$.
By Lemma \ref{leminversemetric}, Lemma \ref{finalset5} and Lemma \ref{lem:4dimset}, $g^{ij}-\delta_{ij}$ also has an expansion in $A$ up to order $-q_{N}$. Since we do not care about the particular form of $U_{q_{N},ij}$, we write
\begin{equation*}
	g^{ij}= \delta_{ij} + U_{q_{N},ij} + O(-q_{N}-\epsilon).
\end{equation*}
 By Remark \ref{struorder} which implies $U_{q_{N},ij}=O(1-n)$ and the fact $\partial_{x_i}O(-p)=O(-(p+1))$ for any $p>0$,
similarly, we summarize the consequence of some direct computation  using Lemma \ref{finalset5} and Lemma \ref{lem:4dimset} as follows
\begin{enumerate}[(A)]
	\item
		\begin{equation*}
 g_{ij},\,g^{ij}=\delta_{ij}+U_{q_{N},ij}+O(-q_{N}-\epsilon),
			\label{metricformal}
		\end{equation*}
		where $U_{q_{N},ij}\in \mbox{SPAN}(A)$ and in particular $U_{q_{N},ij}=O(1-n)$;
	\item
		\begin{equation*}
			\partial_{x_l}g_{ij},\partial_{x_l}g^{ij}=\hat{V}_{q_{N},ij}+O(-q_{N}-1-\epsilon),
			\label{metricformal1}
		\end{equation*}
		where $\hat{V}_{q_{N},ij}\in \mbox{SPAN}(A)$ and $\hat{V}_{q_{N},ij}=O(-n)$;
	\item
		\begin{equation*}
			\partial^2_{kl}g_{ij}=V_{q_{N},ij}+O(-q_{N}-2-\epsilon),
			\label{metricformal2}
		\end{equation*}
		where $V_{q_{N},ij}\in \mbox{SPAN}(A)$ and $V_{q_{N},ij}=O(-n-1)$.
\end{enumerate}

The $RHS$ contains three types of terms, which we study one by one.
Direct calculation using (\ref{metricformal}) and (\ref{metricformal2}) shows that $-(g^{kl} -\delta^{kl})\partial_{kl}(g_{ij})$, which is the first term in the right hand side of (\ref{equricflatnew}), has the form of
\begin{equation}\label{nonlinear1}
 \begin{split}
-(g^{kl} -\delta^{kl})\partial^2_{kl}(g_{ij})&= W_{q_{N},ij}+O(-(q_{N}+n+1)-\epsilon),\\
\text{where} \quad W_{q_{N},ij}&\in \mathcal{L}_{A\times A},
\end{split}
\end{equation}
since $U_{q_{N},ij}$, $V_{q_{N},ij}$ have no constant terms.
In the above computation we used the equality $O(-p_{1})O(-p_{2})=O(-p_{1}-p_{2})$, where $p_{1},p_{2}>0$.
And later we will use it implicitly for many times.

The more difficult part of the right hand side of (\ref{equricflatnew}) is $Q(g,\partial g)$ (see (\ref{eqn:nonliearterm})). We divide it into two parts:
\begin{eqnarray*}
 & &Q_{1,ij}:=-\sum_{k,l,p,q} g^{pq}g_{lj}\partial_{p} g_{ik}\partial_{q} g^{kl},\\
 & &Q_{2,ij}:= -\sum_{k,l,p,q,r,s} 2g_{ik}g_{lj} g^{pq}g^{rs}\Gamma^{k}_{pr}\Gamma^{l}_{qs}.
\end{eqnarray*}
 From (\ref{metricformal}) and (\ref{metricformal1}), we have
\begin{eqnarray*}
& &Q_{1,ij}\\
&=& -\sum_{k,l,p,q} (\delta_{pq}+U_{q_{N},pq}+ O(-q_{N}-\epsilon))(\delta_{lj}+U_{q_{N},lj}+ O(-q_{N}-\epsilon))
\partial_{p} g_{ik}\partial_{q} g^{kl}\nonumber\\
&=&-\sum_{k,l,p,q}(\delta_{pq}+U_{q_{N},pq})(\delta_{lj}+U_{q_{N},lj})\partial_{p} g_{ik}\partial_{q} g^{kl}+O(-(q_{N}+2n)-\epsilon).
 \end{eqnarray*}
 In the right hand of the last equality there are two types of terms except the error term, which are
 \begin{eqnarray*}
 -\sum_{k,l,p,q}\delta_{pq}\delta_{lj}\partial_{p} g_{ik}\partial_{q} g^{kl}&:=&
 \delta\star\delta\star\partial g\star\partial g^{-1},\\
 -\sum_{k,l,p,q}U_{q_{N},pq}\delta_{lj}\partial_{p} g_{ik}\partial_{q} g^{kl}&:=&
 U_{q_{N}}\star\delta\star\partial g\star\partial g^{-1},
 \end{eqnarray*}
 where $\star$ denotes the summation on the repeated indices.
 Notice that (\ref{metricformal1}) indicates that both $ \partial g\star\partial g$ and $\partial g\star\partial g^{-1}$
 have the form $ \hat{V}_{q_{N}}\star\hat{V}_{q_{N}}+O(-(q_{N}+n+1)-\epsilon)$.
 Then it is easy to see from (\ref{metricformal}) and (\ref{metricformal1}) that $Q_{1,ij}$ has the form
 \begin{equation}\label{analysisquadratic1}
 \begin{split}
Q_{1,ij}&= M_{q_{N},ij}+O(-(q_{N}+n+1)-\epsilon),\\
\text{where}&\quad M_{q_{N},ij}\in \mathcal{L}_{A\times A},
\end{split}
 \end{equation}
 since $U_{q_{N},ij},\hat{V}_{q_{N},ij}$ have no constant terms.

Because \begin{eqnarray*}
\begin{split}
\Gamma^{k}_{pr}&=\frac{1}{2}g^{km}(\partial_{r} g_{pm}+\partial_{p} g_{mr}-\partial_{m} g_{pr})\\
&=\frac{1}{2}\sum_{k,m=1}^{n}(\delta^{km}+U_{q_{N},km}+O(-q_{N}-\epsilon))(\partial_{r} g_{pm}+\partial_{p} g_{mr}-\partial_{m} g_{pr}),
\end{split}
\end{eqnarray*}
we know from (\ref{metricformal}) and (\ref{metricformal1}) that
 \begin{equation}\label{gamma}
 \begin{split}
 &\Gamma=(\delta+U_{q_{N}})\star\partial g+O(-q_{N}-n-\epsilon),\\
 &\text{and in particular} \quad\Gamma=O(-n).
 \end{split}
 \end{equation}
  Here we have used notation $\star$.
 For $Q_{2}$, again by using (\ref{metricformal}) and (\ref{gamma}), and if we denote
 $x\star x, \,x\star x\star x,\,\cdots$ by $x^{\star 2},\,x^{\star 3},\,\cdots$ respectively,
  then we have
\begin{eqnarray*}
& &Q_{2,ij}\\
&=&-\sum_{k,l,p,q,r,s}2(\delta_{ik}+U_{q_{N},ik}+ O(-q_{N}-\epsilon))\cdots(\delta^{rs}+U_{q_{N},rs}+ O(-q_{N}-\epsilon))\Gamma^{k}_{pr}\Gamma^{l}_{qs} \nonumber\\
&=&-\sum_{k,l,p,q,r,s}2(\delta_{ik}+U_{q_{N},ik})\cdots(\delta^{rs}+U_{q_{N},rs})\Gamma^{k}_{pr}\Gamma^{l}_{qs}
+O(-(q_{N}+2n)-\epsilon)\\
&=&\delta^{\star 4}\star\Gamma^{\star 2}
+U_{q_{N}}\star\delta^{\star 3}\star\Gamma^{\star 2}+\cdots +U_{q_{N}}^{\star 4}\star\Gamma^{\star 2}+O(-(q_{N}+2n)-\epsilon)\nonumber\\
&=&\delta^{\star 4}(\delta+U_{q_{N}})^{\star 2}\star\partial g^{\star 2}
+U_{q_{N}}\star\delta^{\star 3}\star(\delta+U_{q_{N}})^{\star 2}\star\partial g^{\star 2}
+\cdots\\
& &+U_{q_{N}}^{\star 4}\star(\delta+U_{q_{N}})^{\star 2}\star\partial g^{\star 2}+O(-(q_{N}+2n)-\epsilon).
\end{eqnarray*}
Again because of (\ref{metricformal}) and (\ref{metricformal1}), we get that $Q_{2,ij}$ has the form
\begin{eqnarray}\label{analysisquadratic2}
\begin{split}
Q_{2,ij}&=L_{q_{N},ij}+O(-(q_{N}+n+1)-\epsilon),\\
\text{where}&\quad L_{q_{N},ij}\in\mathcal{L}_{A\times A}.
\end{split}
\end{eqnarray}
Finally we conclude by combining (\ref{equricflatnew}), (\ref{nonlinear1}), (\ref{analysisquadratic1}) and (\ref{analysisquadratic2}) that Lemma \ref{struresult} is valid.
\end{proof}

 Lemma \ref{struresult} is the key to make the bootstrapping argument work,
and together with requirement $l\geq j+1$ in the definition of
$ \mathcal{T}$, we can avoid \textbf{exceptional} terms (see (\ref{exceptionaldef})) when $n>4$.

\begin{prop}\label{pruductpro}
Products like $T_{1}\cdot T_{2}$ ($T_{i}\in\mathcal{T},i=1,2$) contain no \textbf{exceptional} terms
 when $n>4$.
And the equation
 \begin{equation*}
\triangle_{\mathbb{R}^{n}}V=T_{1}\cdot T_{2}
 \end{equation*}
on $ \mathbb{R}^{n}\setminus B_{R}$ has a solution in $\mbox{SPAN}(\mathcal{T})$.
\end{prop}

\begin{proof}
Thanks to (\ref{eqn:decomposition*}),
 \begin{eqnarray}\label{product}
 T_{1}\cdot T_{2}
 &=&r^{\sigma_{1}}G_{m_{1}}r^{\sigma_{2}}G_{m_{2}}\nonumber \\
 &=&r^{\sigma_{1}+\sigma_{2}}(G_{m_{1}+m_{2}}+G_{m_{1}+m_{2}-2}+\cdots).
 \end{eqnarray}
 By definition of $\mathcal{T},$ $\sigma_{1}+\sigma_{2}=2(j_{1}+j_{2})-(n-2)(l_{1}+l_{2})-k_{1}-k_{2},$
  and
  \begin{equation}\label{poweradd}
  k_{1}+k_{2}\geq( m_{1}+m_{2}),\,k_{1}+k_{2}\geq l_{1}+l_{2}\geq j_{1}+j_{2}+2,
  \end{equation}
   therefore we have
  \begin{equation}\label{powercomparison}
   -(\sigma_{1}+\sigma_{2})-(m_{1}+m_{2}+n)\geq2(n-2)+(n-4)(j_{1}+j_{2})-n>0,
  \end{equation}
  when $n\geq5$.
(\ref{powercomparison}) indicates that there are no \textbf{exceptional} terms in (\ref{product}) because $r^{\sigma}G_{m}$ is
exceptional if and only if $\sigma+n+m=0$.

 Now we have proved that every term in the right hand side of (\ref{product}), which takes the form
\begin{equation*}
r^{2(j_{1}+j_{2})-(n-2)(l_{1}+l_{2})-(k_{1}+k_{2})}G_{\beta}:=f_{\beta},
\end{equation*}
is not exceptional,
where $\beta=m_{1}+m_{2},m_{1}+m_{2}-2,\cdots,{m_{1}+m_{2}-2[\frac{m_{1}+m_{2}}{2}]}$.
And from the definition of $\mathcal{T},$ we have
\begin{equation}\label{powcomp2}
\frac{k_{1}+k_{2}-\beta}{2}\in \mathbb{N}\cup \set{0}.
\end{equation}
Then
for each $\beta$, by Proposition \ref{solutionlog} the equation
  \begin{equation*}
\triangle_{\mathbb{R}^{n}}V=f_{\beta}
 \end{equation*}
 has a solution
 \begin{equation*}
V_{\beta}=c_{\beta}r^{2}\cdot f_{\beta},
 \end{equation*}
 where $c_{\beta}$ is a real number.
 Set $j=j_{1}+j_{2}+1,l=l_{1}+l_{2},k=k_{1}+k_{2}$, we know easily that the above solution,
 $$V_{\beta}= c_{\beta}r^{2j-(n-2)l-k}G_{\beta},$$
 lies in $\text{SPAN}(\mathcal{T}),$
 since $k\geq l\geq j+1$ (see (\ref{poweradd})), and $\frac{k-\beta}{2}\in \mathbb{N}\cup \set{0}$ (see (\ref{powcomp2})).
 Thus the second part of the proposition is right.
\end{proof}

When $n=4$, in spite of the existence of log terms, we still have
\begin{prop} \label{product4}
The equation
\begin{equation*}
\triangle_{\mathbb{R}^{n}}V=T_{1}\cdot T_{2}
 \end{equation*}
on $ \mathbb{R}^{n}\setminus B_{R}$ has a solution in $\mbox{SPAN}(\widetilde{\mathcal{T}})$, where $T_{i}\in\widetilde{\mathcal{T}},i=1,2$.

\end{prop}

\begin{proof}
Suppose $T_{j}=r^{\sigma_{j}}(\log r)^{i_{j}}G_{m_{j}}\in \widetilde{\mathcal{T}}$ ($j=1,2$), where $\sigma_{j}=-2l_{j}-k_{j}$, then from (\ref{4set}) we have $ l_{j}\geq1,k_{j}\geq l_{j}+i_{j},$ and $\frac{k_{j}-m_{j}}{2}\in\mathbb{N}\cup\set{0}.$
Next by (\ref{eqn:decomposition*})
\begin{equation}\label{tiltaproduct}
T_{1}\cdot T_{2}=r^{-2l-k}(\log r)^{i}(G_{m}+G_{m-2}+\cdots+G_{m-2[\frac{m}{2}]}),
\end{equation}
where $l=l_{1}+l_{2},k=k_{1}+k_{2},i=i_{1}+i_{2},m=m_{1}+m_{2}$ and $[\frac{m}{2}]$ is the largest integer less than $\frac{m}{2}$.
Obviously from above we have
\begin{eqnarray}\label{coefficientsrelation}
l\geq2, k\geq l+i, \frac{k -m} {2}\in\mathbb{N}\cup\set{0}.
\end{eqnarray}
The relations $ l\geq2$ and $ k\geq l+i$ in (\ref{coefficientsrelation}) indicate that each term of the form $T_{1}\cdot T_{2}$ is in a proper subset of
$\text{SPAN}(\widetilde{\mathcal{T}})$. It can be roughly said that the difference between $k$ and $i$ increases at least by one after doing multiplication.

By Proposition \ref{solutionlog} for each term $r^{-2l-k}(\log r)^{i}G_{s}=f_{s}$ ($s=m,m-2,\cdots,m-2[\frac{m}{2}]$) in (\ref{tiltaproduct}),
\begin{equation*}
\triangle_{\mathbb{R}^{n}}u= f_{s}
\end{equation*}
has a solution of the form
\begin{equation}\label{solution41}
r^{-2l-k+2}(\sum_{j=0}^{i}c_{j}(\log r)^{j})G_{s},c_{j}\in \mathbb{R}
\end{equation}
if $s+n\neq 2l+k$, that is $s+4\neq 2l+k$;
and of the type of
\begin{equation}\label{solution42}
r^{-2l-k+2}(\sum_{j=0}^{i+1}c_{j}(\log r)^{j})G_{s},c_{j}\in \mathbb{R}
\end{equation}
if $s+4= 2l+k$.
In fact this can happen only when $l=2$ and $k=s$, since $l\geq2, k\geq s .$
One observation is that when solving the Poisson equations above, the power of $\log r$ increases at most one.

(\ref{coefficientsrelation}), (\ref{solution41}) and (\ref{solution42}) show that
\begin{equation*}
\triangle_{\mathbb{R}^{n}}u= T_{1}\cdot T_{2}
\end{equation*}
has a solution which is a finite linear combination of terms taking the form of
\begin{equation*}
r^{-2l-k+2}(\log r)^{j}G_{s},
\end{equation*}
where $j\leq i+1,l\geq2,k\geq i+l,$ and $ i,j,l,k,s,\frac{k-s}{2}\in\mathbb{N}\cup\set{0}.$
Let $\tilde{l}=l-1$, thereby $ \tilde{l}\geq1$. Next it's straightforward to check each term $r^{-2\tilde{l}-k}(\log r)^{j}G_{s}$ in the last formula
satisfies the conditions
\begin{eqnarray*}
\tilde{l}\geq1, k\geq j+\tilde{l},\quad j,\tilde{l},k,s,\frac{k-s}{2}\in\mathbb{N}\cup\set{0}
\end{eqnarray*}
 in (\ref{4set}), thus they are all in $\widetilde{\mathcal{T}}$. This results in the validity of the proposition.
\end{proof}

\section{Proof Of The Main Results}\label{sec:pro}

With all the preparations, we are at the position to prove our main results. First we give the proof of Theorem \ref{MAINTHM}. The second part is from Remark \ref{remaftermain}, so we only need to prove the first part.

\begin{proof}
We argue by induction.
We first prove the case $n>4$.
 At the beginning, due to Remark \ref{rem1thexpansion} we have an expansion of the form $g_{ij}-\delta_{ij}=a_{1,ij}r^{1-n}G_{1}+O(1-n-\epsilon)$.
 Notice that $r^{-1-(n-2)}G_{1}\in\mathcal{T}_{1}\subset \mathcal{T}$, hence $g_{ij}-\delta_{ij}$ has an expansion up order $-q_{1}=1-n $ in the sense of Definition \ref{defexpansionintr}, this is the starting point of the induction. Here we recall we have defined that $- q_{N}=N(1-n)$ at the beginning of Subsection \ref{structure}.

 Next assume $g_{ij}-\delta_{ij}$ has an expansion up order $-q_{N}, N\in\mathbb{N}$ with term in
$\mathcal{T}$ in the sense of Definition \ref{defexpansionintr}, then by using Lemma \ref{struresult}, Proposition \ref{pruductpro} and Lemma \ref{corpoissonestimate4}, we know the equation
 \begin{equation*}
\triangle _{\mathbb{R}^{n}}u=RHS_{ij}=RHS_{q_{N},ij}+O(-q_{N}-n-1-\epsilon)
\end{equation*}
on $ \mathbb{R}^{n}\setminus B_{R}$
has a solution $\hat{U}_{q_{N+1},ij} +O(-q_{N}-n+1-\epsilon)$, where $\hat{U}_{q_{N+1},ij}\in \text{SPAN}(\mathcal{T})$.
On account of (\ref{equricflatnew}) and the fact $g_{ij}-\delta_{ij}\in \mathcal{C}^{\infty}_{1-n}$,
\begin{equation*}
 g_{ij}-\delta_{ij}-(\hat{U}_{q_{N+1},ij} +O(-q_{N}-n+1-\epsilon)):=H_{q_{N+1}}
 \end{equation*}
 is of order $r^{1-n}$ and is harmonic with respect to $\triangle _{\mathbb{R}^{n}}$.
 So $ H_{q_{N+1}}\in \text{SPAN}(\mathcal{T}_{1})\subset \text{SPAN}(\mathcal{T})$ by Lemma \ref{lem:harmonic expansion}.
 Therefore if we set
 \begin{equation*}
 U_{q_{N+1},ij}:=\hat{U}_{q_{N+1},ij} +H_{q_{N+1}},
 \end{equation*}
 then we get a better expansion in terms of $\mathcal{T}$ up to order $-q_{N+1}$ in the sense of Definition \ref{defexpansionintr} for $g_{ij}-\delta_{ij}$, which is
 \begin{equation*}
 g_{ij}-\delta_{ij}=U_{q_{N+1},ij}+O(-q_{N+1}-\epsilon).
 \end{equation*}
 Here we have used the definition $-q_{N}=N(1-n). $
 Finally we conclude that $g_{ij}-\delta_{ij}$ has an expansion up to any order in terms of
$\mathcal{T}$ in the sense of Definition \ref{defexpansionintr} by taking larger and larger $N$.

When $n=4$, since $r^{-1-(n-2)}G_{1}\in\mathcal{T}_{1}\subset \widetilde{\mathcal{T}}$, we still have
 the starting point of the induction.
 With $\widetilde{\mathcal{T}}$ in place of $\mathcal{T}$, and Proposition \ref{product4} in place of
 Proposition \ref{pruductpro}, the induction argument can work just as in the case $n>4$.

\end{proof}

Before the proof of Corollary \ref{coreven}, we recall some basic properties of the Kelvin transformation. Obviously it is a linear operator.
And it can be checked directly that
\begin{equation}
	\begin{split}
K[1]&=|x|^{2-n},\\
K[r^\sigma G_{m}]&=r^{2-n-\sigma}G_{m},\\
K[O(r^{-q-\epsilon})]&=O(r^{q+2-n+\epsilon}).
	\end{split}
	\label{kelvin}
\end{equation}
First from (\ref{kelvin}) and Remark \ref{remaftermain}, we can deduce that for all $n\geq 4$,
$K[g_{ij}-\delta_{ij}]$ has an expansion of the form
\begin{equation}\label{firsthar}
K[g_{ij}-\delta_{ij}]=H_{n-1} +o(r^{n-1}),
\end{equation}
where $H_{n-1}$ is a harmonic polynomial (without constant term) of degree $\leq n-1 .$
In particular, $K[g_{ij}-\delta_{ij}]$ is $\mathcal{C}^{n-1}$ differentiable on $B_{\frac{1}{2R}}$.

From now on assume that $n>4$ and $g_{ij}$ has the expansion up order $q$ as given in Theorem \ref{MAINTHM}, then
\begin{eqnarray*}
K[g_{ij}]&=&K[\delta_{ij}]+K[U_{q,ij}]+K[O(-q-\epsilon)]\\
&=&|x|^{2-n}\delta_{ij}+K[U_{q,ij}]+O(r^{q+2-n+\epsilon}),
\end{eqnarray*}
where $q\in\mathbb{N}$, $q\geq n-1$ and $U_{q,ij}\in \text{SPAN}(\mathcal{T})$.

Now we prove Corollary \ref{coreven}.

\begin{proof}
Because $n$ is even, for every $r^\sigma G_{m}\in \mathcal{T}$, we have $2-n-\sigma-m$ is even
and $r^{2-n-\sigma-m}=(z^{2}_{1}+\cdots+z^{2}_{n})^{\frac{2-n-\sigma-m}{2}}$ is a polynomial, which implies that
$r^{2-n-\sigma}G_{m}$ is also a polynomial by the definition of $G_m$.
So the second equation in (\ref{kelvin}) implies that $K[U_{q,ij}]$ is polynomial of degree $\leq q+2-n$.
And (\ref{firsthar}) infers $K[g_{ij}-\delta_{ij}](0)=0$, thus the proof is finished.
\end{proof}

\bibliographystyle{plain}
\bibliography{foo}

\end{document}